\documentclass[12pt,a4paper]{article}
\usepackage{amsmath,amsthm,amssymb,amsfonts,color,graphicx}

\newcommand{\R}{{\mathbb{R}}}          % \R       = reais
\newcommand{\lrr}{\longrightarrow}

\newcommand{\calR}{{{\cal R}^\xi}}             %

\newcommand{\na}{{\nabla}}
\newcommand{\nag}{{\nabla^g}}

\newcommand{\End}[1]{{\mathrm{End}}\,{#1}}

\newcommand{\ric}{{\mathrm{ric}}}
\newcommand{\grad}{{\mathrm{grad}}\,}
\newcommand{\dx}{{\mathrm{d}}}
\newcommand{\dxna}{{\mathrm{d}^\na}}

\newcommand{\cinf}[1]{{\mathrm{C}}^\infty_{#1}}

\newcommand{\expo}{{\mathrm{e}}}

\newtheorem{teo}{Theorem}[section]

\newtheorem{coro}{Corollary}[section]
\newtheorem{prop}{Proposition}[section]

\newenvironment{meuenumerate}
{\begin{enumerate}
  \setlength{\itemsep}{2.5pt}
  \setlength{\parskip}{-1pt}
  \setlength{\parsep}{-1pt}}
{\end{enumerate}}

\def\cyclic{\mathop{\kern0.9ex{{+}
\kern-2.2ex\raise-.28ex\hbox{\Large\hbox{$\circlearrowright$}}}}\limits}

\title{Curvatures of weighted metrics on tangent sphere bundles}

\author{R. Albuquerque\footnote{{\texttt{rpa@uevora.pt}}\ ,\ \ \ Departamento de Matem\'atica da Universidade de \'Evora and Centro de Investiga\c c\~ao em Matem\'atica e Aplica\c c\~oes (CIMA), Rua Rom\~ao Ramalho, 59, 671-7000 \'Evora, Portugal.}}

\begin{document}

%\begin{color}{DarkBlue}
%\begin{color}{black}

\maketitle

%\date{\today} %{1 dezembro 2009}

\markright{\sl\hfill  R. Albuquerque \hfill}

\begin{abstract}

We determine the curvature equations of natural metrics on tangent
bundles and radius $r$ tangent sphere bundles $S_rM$ of a Riemannian
manifold $M$. A family of positive scalar curvature metrics on $S_rM$ is found,
for any $M$ with bounded sectional curvature and any chosen constant $r$.

\end{abstract}

\vspace*{4mm}


\vspace*{2mm}


\vspace*{10mm}

The author acknowledges the support of Funda\c{c}\~{a}o Ci\^{e}ncia e Tecnologia, Portugal, through Centro de Investiga\c c\~ao em Matem\'atica e Aplica\c c\~oes da Universidade de \'Evora (CIMA-UE) and the sabbatical grant SFRH/BSAB/895/2009.

\vspace*{12mm}

\section{Introduction}

This article continues the study of some structures which identify the tangent
sphere bundles $S_rM=\{u\in TM:\ \|u\|=r\}$ of a Riemannian manifold $(M,g)$
with variable radius and weighted Sasaki metric. We use the same notation from
\cite{Alb3}.

Throughout, we assume that $M$ is an $m$-dimensional manifold with a Riemannian
metric $g$ and a compatible metric connection $\na$ on $M$. The latter induces a
splitting of $TTM=H\oplus V$ with both $H,V$ parallel and isometric to
$\pi^*TM$. We have a map $\theta\in\End{TTM}$, which identifies $H$ with $V$,
sends $V$ to 0 and is parallel for $\na^*=\pi^*\na\oplus\pi^*\na$. The manifold
$TM$ is endowed with a canonical vertical vector field $\xi$, defined by
$\xi_u=u$. It is known as the spray of the connection since $\pi^*\na_X\xi=X^v$
and this projection has kernel $H$. 

%We assume $\na$ is the Levi-Civita connection. 
We continue our study assuming metrics of the kind $g^{f_1,f_2}=f_1\pi^*g\oplus
f_2\pi^*g$ on $H\oplus V$, where $f_1,f_2$ are given by
\begin{equation}
 f_1=\expo^{2\varphi_1},\qquad f_2=\expo^{2\varphi_2},
\end{equation}
for some functions $\varphi_1,\varphi_2$ on $M$. Obviously we let these functions be composed with $\pi$ when considered on the manifold $TM$. Recall the well known Sasaki metric is just $g^S=g^{1,1}=\langle\cdot,\cdot \rangle$ with $H$ induced by the Levi-Civita connection. We remark the
addition of a third component $f_3\mu\otimes\mu$, where
$\mu=(\theta^t\xi)^\flat$, gives a metric with interesting properties on
$S_rM$, rather than the more studied Cheeger-Gromov metric.

We treat all vectors equally and use canonical projections $X=X^v+X^h$ when
necessary, since we do not recur to lifts of tangent vectors on $M$ to either
sections of $H$ or $V$. We wish to concentrate on tensors defined on $TM$.
Notice the holonomy Lie algebra of any of the metrics above remains unknown
in general, even if $M$ is any irreducible Riemannian symmetric space. Our main
objective here is to envisage a solution to that problem and so we compute several curvature formulas.

The geometry of tangent bundles has had much attention in the past and the
Riemannian curvature of the Sasaki metric has been found (cf. the references
in \cite{Alb3,Blair,Munteanu}). Regarding the radius $r$ tangent sphere bundle with the induced metric from $g^{f_1,f_2}$ we achieve in Theorem \ref{enfimcurvaturaescalarpositiva} a generalisation of a result from \cite{KowSek2}: if $M$ has $\dim\geq3$ and bounded sectional curvature, and
$f_1$ is sufficiently large or $f_2$ is sufficiently small, with both constant,
then $S_rM$ has positive scalar curvature. 

Our purpose with this study is also towards the geometry of the so called gwistor
bundle, which is the natural $G_2$-structure existing on $S_1M$ for any oriented
Riemannian 4-manifold.

Parts of this article were written during a sabbatical leave of the author
at the Philipps-Universit\"at Mathematics Department, Marburg. He
wishes to thank their great hospitality and the excellent time spent there.

\subsection{Computing the curvature of $TM$}
\label{CurvatureofTM}

Let $\na=\nag$ denote the Levi-Civita connection of $M$. As one of the few
cases one can cope with, we study the curvature of $G=g^{f_1,f_2}$
where $f_2=\expo^{2\varphi_2}$ is a function on $M$ and $f_1$ is a constant. We
define $\delta=\frac{f_2}{f_1}$. 

Recall from \cite[Theorem 5.2]{Alb3} that the Levi-Civita connection of the tangent bundle is given by determining first
\begin{equation}
  \na^{*,f_2,'}_XY^v=\na^*_XY^v+X(\varphi_2)Y^v,
\end{equation}
\begin{equation}
  D^*=\na^*\oplus\na^{*,f_2,'} \hspace{1cm}\mbox{on}\ \ H\oplus V=TTM,
\end{equation}
\begin{equation}
B(X,Y)=Y(\varphi_2)X^v-\delta\langle X^v,Y^v\rangle\grad\varphi_2,
\end{equation}
\begin{equation}
 \langle A_XY,Z\rangle=\frac{\delta}{2}(\langle\calR(X,Z),Y\rangle 
+\langle\calR(Y,Z),X\rangle).
\end{equation}
The first connection is metric on the vector bundle $V$. The tensor $\calR$ is
given by $\calR(X,Y)=\pi^*R^\na(X,Y)\xi$ and finally $\forall X,Y\in\Gamma(TM,H\oplus V)$, we have
\begin{equation}
 \na^G_XY=D^*_XY-\frac{1}{2}\calR(X,Y)+A(X,Y)+B(X,Y).
\end{equation}

We recall, for a moment, that if $\na'=\na+C$ and $\na$ are two connections on a
vector bundle $L$, hence with $C\in\Omega^1(\End{L})$, then
\begin{equation}
 R^{\na'}=R^\na+\dx^\na C+C\wedge C
\end{equation}
where
\begin{equation}
  \dx^\na C(X,Y)=\na_XC_Y-\na_YC_X-C_{[X,Y]}
\end{equation}
and
\begin{equation}
 (C\wedge C)(X,Y)Z=C(X,C(Y,Z))-C(Y,C(X,Z))
\end{equation}
with $X,Y$ vector fields and $Z$ a section of $L$.

Now, we have to compute several $\dx^{\na}$ derivatives of our structure, where
$\na=\na^*\oplus\na^*$ respecting the splitting $H\oplus V$. Recall the formula
already implicitly used, $R^{\na^*}=\pi^*R^\na$, for this is a tensor. Assuming
the reader is by now familiar with the notation, we shall let fall the asterisk
wherever possible and abbreviate $R^\na=R$. 

Let $A^{\na_X\calR}$ be defined (in the same way as the tensor $A$ is defined):
\begin{equation}
\langle A^{\na_X\calR}(Y,Z),W\rangle=\frac{\delta}{2}(\langle(\na_X\calR)(Y,W),Z\rangle 
+\langle(\na_X\calR)(Z,W),Y\rangle).
\end{equation}
Again we have the properties 
\begin{equation}
 \na_X\calR(Y,Z)=\na_X\calR(Y^h,Z^h)\ \in V, 
\end{equation}
\begin{equation}
 A^{\na_X\calR}(X,Y)=A^{\na_X\calR}(X^h,Y^v)+A^{\na_X\calR}(X^v,Y^h)\ \in H.
\end{equation}
\begin{prop}\label{partesdacurvatura}
We have:
 \begin{meuenumerate}
  \item $R^{\na^{*,f_2,'}}=R$.
\item $(\na_{X}\calR)(Y,Z)=(\na_{X^h}R)(Y,Z)+R(Y,Z)X^v$.
\item $\dxna\calR(X,Y)Z=(\na_XR)(Y,Z)\xi-(\na_YR)(X,Z)\xi+R(Y,Z)X^v-R(X,Z)Y^v$.
\item $\dxna A(X,Y)Z=(\dx\varphi_2\wedge A)(X,Y)Z-A^{\na_Y\calR}_XZ+A^{\na_X\calR}_YZ+A(\calR(X,Y),Z)$.
 \end{meuenumerate}
\end{prop}
\begin{proof}
 1. The connection is $\na_XY+\dx\varphi_2(X)Y$. Thence $\dxna(\dx\varphi_2.1)=\dx\dx\varphi_2.1=0$. And clearly
\[  \dx\varphi_2.1\wedge\dx\varphi_2.1=\dx\varphi_2\wedge\dx\varphi_2.1=0.         \]
2. For any vector fields:
\begin{eqnarray*}
 \na_X\calR\,(Y,Z)&=& \na^*_X(\pi^*R(Y,Z)\xi)-\pi^*R(\na^*_XY,Z)\xi-\pi^*R(Y,\na^*_XZ)\xi\\
  &=& \pi^*(\na_{\dx\pi X}R)(Y,Z)\xi+R(Y,Z)\na_X\xi\\
  &=& (\na_{X^h}R)(Y,Z)\xi+R(Y,Z)X^v
\end{eqnarray*}
since we have the identity $\na_X\xi=X^v$.\\
3. Since $\calR_X=\calR_{X^h}$ and $\pi^*T^\na=0$, we have 
\begin{eqnarray*}
{\lefteqn{\dxna\calR(X,Y)Z \ =}}\\
 &=& (\na_X\calR_Y-\na_Y\calR_X-\calR_{[X,Y]})Z\\
 &=& \na_X(R(Y,Z)\xi)-R(Y,\na_XZ)\xi-\na_Y(R(X,Z)\xi)+R(X,\na_YZ)\xi\\
 & &\hspace*{3cm}-R(\na_XY,Z)\xi+R(\na_YX,Z)\xi\\
&=& (\na_XR)(Y,Z)\xi-(\na_YR)(X,Z)\xi+R(Y,Z)\na_X\xi-R(X,Z)\na_Y\xi\\
&=& \na_X\calR\,(Y,Z)-\na_Y\calR\,(X,Z).
\end{eqnarray*}
4. First we find
\begin{eqnarray*}
\lefteqn{\langle\na_X(A(Y,Z)),W\rangle\ =}\\
 &=& X(\langle A(Y,Z),W\rangle)-\langle A(Y,Z),\na_XW\rangle\\
&=& \frac{1}{2f_1}(X(f_2))(\langle \calR(Y,W),Z\rangle+\langle\calR(Z,W),Y\rangle)\\
& &+\frac{f_2}{2f_1}\bigl(\langle\na_X(\calR(Y,W)),Z\rangle+\langle\calR(Y,W),\na_XZ\rangle+\langle\na_X(\calR(Z,W)),Y\rangle+\\
& &+\langle\calR(Z,W),\na_XY\rangle-\langle\calR(Y,\na_XW),Z\rangle-\langle\calR(Z,\na_XW),Y\rangle\bigr)\\
&=& \langle X(\varphi_2)A(Y,Z),W\rangle+\frac{f_2}{2f_1}\bigl(\langle(\na_X\calR)(Y,W)+\calR(\na_X Y,W),Z\rangle+\\
& & \langle\calR(Y,W),\na_XZ\rangle+\langle(\na_X\calR)(Z,W)
+\calR(\na_X Z,W),Y\rangle+\langle\calR(Z,W),\na_XY\rangle\bigr)\\
&=& \langle X(\varphi_2)A(Y,Z)+A^{\na_X\calR}(Y,Z)+A(\na_XY,Z)+A(Y,\na_XZ),W\rangle.
\end{eqnarray*}
Recalling the torsion of $\na^*$ is $\calR$, cf. \cite[Proposition
5.1]{Alb3}, we then have
\begin{eqnarray*}
\lefteqn{\dxna A(X,Y)Z\ =}\\
&=& (\na_XA_Y)Z-(\na_YA_X)Z-A_{[X,Y]}Z\\
&=&\na_X(A(Y,Z))-A(Y,\na_XZ)-\cdots\\
&=& X(\varphi_2)A(Y,Z)+A^{\na_X\calR}(Y,Z)+A(\na_XY,Z)-Y(\varphi_2)A(X,Z)\\
& & -A^{\na_Y\calR}(X,Z)-A(\na_YX,Z)-A(\na_XY-\na_YX-\calR(X,Y),Z)\\
&=& \dx\varphi_2\wedge A(X,Y)Z+A^{\na_X\calR}(Y,Z)-A^{\na_Y\calR}(X,Z)+A(\calR(X,Y),Z)
\end{eqnarray*}
as we wished.
\end{proof}
In a very similar computation as the above we find:
\begin{prop}
The $B$ tensor satisfies
 \begin{equation}
  \begin{split}
  \dxna B(X,Y)Z\ =\ \langle\na_X\grad\varphi_2,Z\rangle Y^v-\langle\na_Y\grad\varphi_2,Z\rangle X^v +Z(\varphi_2)\calR(X,Y) \hspace{1cm}\\
-\delta\bigl(2X(\varphi_2)\langle Y^v,Z^v\rangle-2Y(\varphi_2)\langle X^v,Z^v\rangle-\langle\calR(X,Y),Z\rangle\bigr.\grad\varphi_2\\
 \bigl.-\langle Y^v,Z^v\rangle\na_X\grad\varphi_2+\langle X^v,Z^v\rangle\na_Y\grad\varphi_2\bigr).
\hspace*{3.2cm}
  \end{split}
 \end{equation}
\end{prop}

Now, we want to compute the curvature of $\na^G$. As the reader might see, the development of $\dxna C+C\wedge C$ is quite long when $C=\dx\varphi_2.1^v-\frac{1}{2}\calR+A+B$. So we shall proceed with two particular cases. The first is well at hand. The second is in the next section.
\begin{teo}
 Suppose $f_1>0$ is a constant, $f_2=\expo^{2\varphi_2}$ and the connection $\na$ is flat, so that
\begin{equation}
\na^G_XY=\na_XY+X(\varphi_2)Y^v+Y(\varphi_2)X^v-\delta\langle X^v,Y^v\rangle\grad\varphi_2.
\end{equation}
Then the Riemannian curvature tensor of $TM$ with metric $G=g^{f_1,f_2}$ is given by
\begin{equation}
  \begin{split}
  R^G(X,Y)Z\ =\ \bigl(X(\varphi_2)Z(\varphi_2)+\delta\epsilon^2\langle X^v,Z^v\rangle+\langle\na_X\grad\varphi_2,Z\rangle\bigr)Y^v\\
-\bigl(Y(\varphi_2)Z(\varphi_2)+\delta\epsilon^2\langle Y^v,Z^v\rangle+\langle\na_Y\grad\varphi_2,Z\rangle\bigr)X^v \hspace{.5cm}\\
-\delta\bigl(X(\varphi_2)\langle Y^v,Z^v\rangle- Y(\varphi_2)\langle X^v,Z^v\rangle\bigr)\grad\varphi_2  \hspace*{1.2cm}\\
 -\delta\langle Y^v,Z^v\rangle\na_X\grad\varphi_2+\delta\langle X^v,Z^v\rangle\na_Y\grad\varphi_2
\hspace*{.3cm}
  \end{split}
\end{equation}
where $\epsilon=\|\grad\varphi_2\|$.
\end{teo}
\begin{proof}
 After some computations we find
\[  B\wedge B(X,Y)Z=\delta\epsilon^2(\langle X^v,Z^v\rangle Y^v-\langle Y^v,Z^v\rangle X^v)  \]
and
\begin{eqnarray*} 
C\wedge C(X,Y)Z &=& (\dx\varphi_2.1^v\wedge B+B\wedge\dx\varphi_2.1^v+B\wedge B)(X,Y)Z\\
 &=& X(\varphi_2)Z(\varphi_2)Y^v-Y(\varphi_2)Z(\varphi_2)X^v+Y(\varphi_2)B(X,Z^v)\\
& & \ \ -X(\varphi_2)B(Y,Z^v)+B\wedge B(X,Y)Z\\
 & = & X(\varphi_2)(Z(\varphi_2)Y^v+\delta\langle Y^v,Z^v\rangle\grad\varphi_2)\\
& & \ \ -Y(\varphi_2)(Z(\varphi_2)X^v+\delta\langle X^v,Z^v\rangle\grad\varphi_2)+B\wedge B(X,Y)Z.
\end{eqnarray*}
Adding to $\dxna C=\dxna B$ above, we deduce $R^G=\dxna C+C\wedge C$.
\end{proof}
The case when $\grad\varphi_2$ is parallel may be further  developed. Straightforward computation yields the following result.
\begin{coro}
Suppose $(M,g)$ is a flat Riemannian manifold and the function $f_2$ verifies $\na\dx\varphi_2=0$. Then the sectional curvature of the metric $G=g^{f_1,f_2}$ on a plane $\Pi$ spanned by the orthonormal basis $X,Y$ is
\begin{equation}
 \begin{split}
%\begin{center}
 k(\Pi) \ =\ G(R^G(X,Y)Y,X)\hspace{6.8cm}\\
= -f_2\epsilon^4\|bX^v-aY^v\|^2-f_2\epsilon^2\delta(\|X^v\|^2\|Y^v\|^2-\langle X^v,Y^v\rangle^2), 
 \end{split}
\end{equation}
where $X=a\,\grad\varphi_2+X'+X^v,\ Y=b\,\grad\varphi_2+Y'+Y^v$ and $X',Y'\in H\cap(\grad\varphi_2)^\perp,\ \,a,b\in\R$. In particular, $k(\Pi)\leq0$. 
\end{coro}
Hence on points $x$ where $\grad\varphi_2\neq0$ the fibres $T_xM$ are hyperbolic totally geodesic submanifolds.

In the previous conditions, we observe that the equations of a geodesic curve $\Theta$ in $TM$ appear as:
\begin{equation}
 \left\{\begin{array}{l}
 \na_{\dot{\Theta}}{\dot{\Theta}}^h-f_2\langle{\dot{\Theta}}^v,{\dot{\Theta}}^v\rangle\grad\varphi_2=0\\
\na_{\dot{\Theta}}{\dot{\Theta}}^v+2{\dot{\Theta}}(\varphi_2){\dot{\Theta}}^v=0.
\end{array}\right.
\end{equation}
So it would be interesting at least in this case to solve the problem of knowing when is $\na^G$ complete. (The completeness of a pull-back connection seems to be an open problem.)

If $M$ is a simply connected flat Riemannian manifold and $\na^G$ is a complete
connection, then $TM$ is very close to being a Stein manifold. To apply a famous
result of Wu, \cite{Wu}, we would need $TM$ to be K\"ahler with $k\leq0$, but
then we are asking $f_2$ to be a constant by \cite[Corollary 6.3]{Alb3}.

\subsection{Curvature of $g^{f_1,f_2}$ with $f_1,f_2$ constants}
\label{CurvatureofTMcontinued}

The second particular situation we must try to investigate is when $f_2$ is a constant. So we continue with $\na=\nag$ the Levi-Civita connection of $M$. We may write simply
\begin{equation}
 \na^G=\na+C\qquad\quad\mbox{with}\qquad\quad C=-\frac{1}{2}\calR+A.
\end{equation}
The connection $D^*=\na^*\oplus\na^*$, so we write it as $\na$. Since $\calR\wedge\calR=0$, the curvature of $G$ is 
\begin{equation}
 R^G=R^\na-\frac{1}{2}\dxna\calR+\dxna A
-\frac{1}{2}\calR\wedge A-\frac{1}{2}A\wedge\calR+A\wedge A.
\end{equation}
Notice $R^\na$ stands for $R^{\na^*}\oplus R^{\na^*}$. Some parts of the tensor
$R^G$ were computed in Proposition \ref{partesdacurvatura}, namely those
involving $\dxna$. Now
\begin{equation}
 \begin{split}
 \dxna\calR(X,Y)Z = (\na_XR)(Y,Z)\xi-(\na_YR)(X,Z)\xi+R(Y,Z)X^v-R(X,Z)Y^v \\
=(\na_X\calR)(Y,Z)-(\na_Y\calR)(X,Z),\hspace{2cm}
 \end{split}
\end{equation}
\begin{equation}
 \dxna A(X,Y)Z = -A^{\na_Y\calR}(X,Z)+A^{\na_X\calR}(Y,Z)+A(\calR(X,Y),Z).
\end{equation}
The others parts do not simplify nor cancel each other, as the reader may notice reading their nature in $H\oplus V$.

Let $e_1,\ldots,e_{m}$ be a real $g$-orthonormal basis of $TM$ at a given point. This is immediately lifted to $H$ and then to $V$ by $\theta$, giving a $g^S$-orthonormal basis. Writing
\begin{equation}
 A(X,Y)=\sum\langle A(X,Y),e_i\rangle e_i=\frac{\delta}{2}\sum(\langle\calR(X,e_i),Y\rangle+\langle
\calR(Y,e_i),X\rangle)e_i,
\end{equation}
we have the Gauss-Codazzi type equations
\begin{equation}
 \begin{split}
 -\frac{1}{2}\calR\wedge A(X,Y)Z = -\frac{1}{2}\calR(X,A(Y,Z))+\frac{1}{2}\calR(Y,A(X,Z))\hspace{2cm}\\
= -\frac{\delta}{4}\sum_j\bigl((\langle\calR(Y,e_j),Z\rangle+\langle\calR(Z,e_j),Y\rangle)
\calR(X,e_j) \quad\\
-(\langle\calR(X,e_j),Z\rangle+\langle\calR(Z,e_j),X\rangle)\calR(Y,e_j)\big),
\end{split}
\end{equation}
\begin{equation}
 \begin{split}
-\frac{1}{2} A\wedge\calR(X,Y)Z = -\frac{1}{2}A(X,\calR(Y,Z))+\frac{1}{2}A(Y,\calR(X,Z))\hspace{2cm}\\
= -\frac{\delta}{4}\sum_i(\langle\calR(X,e_i),\calR(Y,Z)\rangle-\langle\calR(Y,e_i),\calR(X,Z)\rangle)e_i \\
\end{split}
\end{equation}
and
\begin{equation}
 \begin{split}
 A\wedge A(X,Y)Z = A(X,A(Y,Z))-A(Y,A(X,Z))\hspace{3.5cm}\\
= \frac{\delta}{2}\sum_i(\langle\calR(A(Y,Z),e_i),X\rangle-\langle\calR(A(X,Z),e_i),Y\rangle)e_i \quad\\
= \frac{\delta^2}{4}\sum_{i,j}^{m}\bigl((\langle\calR(Y,e_j),Z\rangle+\langle\calR(Z,e_j),Y\rangle)
\langle\calR(e_j,e_i),X\rangle \quad\\
-(\langle\calR(X,e_j),Z\rangle+\langle\calR(Z,e_j),X\rangle)\langle\calR(e_j,e_i),Y\rangle\bigr)e_i.
\end{split}
\end{equation}
Also $A(X,\calR(Y,Z))=\frac{\delta}{2}\sum\langle\calR(X,e_i),\calR(Y,Z)\rangle e_i$. Now we have
\begin{equation}\label{curvaturaRGhhh}
 \begin{split}
  R^G(X^h,Y^h)Z^h=R(X^h,Y^h)Z^h-\frac{1}{2}(\na_{X^h}\calR)(Y^h,Z^h)
+\frac{1}{2}(\na_{Y^h}\calR)(X^h,Z^h)+  \hspace{0.5cm}\\
+A(\calR(X^h,Y^h),Z^h)-\frac{1}{2}A(X^h,\calR(Y^h,Z^h))+\frac{1}{2}A(Y^h,\calR(X^h,Z^h)),
 \end{split}
\end{equation}
\begin{equation}\label{formula70epoucos}
 \begin{split}
  R^G(X^v,Y^h)Z^h= \hspace{5cm}\\
=-\frac{1}{2}(\na_{X^v}\calR)(Y^h,Z^h)
-A^{\na_{Y^h}\calR}(X^v,Z^h)+\frac{\delta}{4}\sum\langle\calR(Z^h,e_j),X^v\rangle\calR(Y^h,e_j)\\
=-\frac{1}{2}R(Y^h,Z^h)X^v
-A^{\na_{Y^h}\calR}(X^v,Z^h)+\frac{\delta}{4}\sum\langle\calR(Z^h,e_j),X^v\rangle\calR(Y^h,e_j),
 \end{split}
\end{equation}
\begin{equation}
 \begin{split}
  R^G(X^v,Y^h)Z^v=A^{\na_{X^v}\calR}(Y^h,Z^v)
+\frac{\delta^2}{4}\sum\langle\calR(Y^h,e_j),Z^v\rangle\langle\calR(e_j,e_i),X^v\rangle e_i,
\hspace{1cm}
 \end{split}
\end{equation}
\begin{equation}
 \begin{split}
  R^G(X^h,Y^h)Z^v=R(X^h,Y^h)Z^v-A^{\na_{Y^h}\calR}(X^h,Z^v)+ A^{\na_{X^h}\calR}(Y^h,Z^v)+
\hspace{1.4cm} \\
+\frac{\delta}{4}\sum\bigl(\langle\calR(X^h,e_j),Z^v\rangle\calR(Y^h,e_j)
-\langle\calR(Y^h,e_j),Z^v\rangle\calR(X^h,e_j)\bigr),
 \end{split}
\end{equation}
\begin{equation}\label{curvaturaRGvvh}
 \begin{split}
  R^G(X^v,Y^v)Z^h=-A^{\na_{Y^v}\calR}(X^v,Z^h)+A^{\na_{X^v}\calR}(Y^v,Z^h)+\hspace{4cm} \\
+\frac{\delta^2}{4}\sum\bigl(\langle\calR(Z^h,e_j),Y^v\rangle\langle\calR(e_j,e_i),X^v\rangle 
-\langle\calR(Z^h,e_j),X^v\rangle\langle\calR(e_j,e_i),Y^v\rangle\bigr) e_i 
 \end{split}
\end{equation}
and, clearly, $R^G(X^v,Y^v)Z^v=0$.

The simplification in formula (\ref{formula70epoucos}) is due to property 2 in
Proposition \ref{partesdacurvatura}.
In order to find the Ricci curvature of $G$ we let $R^G(X,Y,Z,W)$ denote the 4-tensor $G(R^G(X,Y)Z,W)$. The same we agree in denoting $R$ with the metric $g$. We only need
\begin{equation}
 \begin{split}
  R^G(X^h,Y^h,Y^h,W^h)  \hspace{5cm}\\
= f_1R(X^h,Y^h,Y^h,W^h)+f_1\langle A(\calR(X^h,Y^h),Y^h),W^h\rangle
+\frac{f_1}{2}\langle A(Y^h,\calR(X^h,Y^h)),W^h\rangle\\
=f_1R(X^h,Y^h,Y^h,W^h)+\frac{f_2}{2}\langle\calR(Y^h,W^h),\calR(X^h,Y^h)\rangle+
\frac{f_2}{4}\langle\calR(Y^h,W^h),\calR(X^h,Y^h)\rangle \\
=f_1R(X^h,Y^h,Y^h,X^h)+\frac{3}{4}f_2\langle\calR(Y^h,W^h),\calR(X^h,Y^h)\rangle,
 \end{split}
\end{equation}
\begin{equation} \label{formula70etal}
 \begin{split}
  R^G(X^h,Y^v,Y^v,W^h)  \hspace{5cm}\\
=-f_1\langle A^{\na_{Y^v}\calR}(X^h,Y^v),W^h\rangle
-\frac{f_1\delta^2}{4}\sum_{i,j=1}^m\langle\calR(X^h,e_j),Y^v\rangle\langle\calR(e_j,e_i),Y^v\rangle\langle e_i,W^h\rangle\\
=-\frac{f_2}{2}\langle(\na_{Y^v}\calR)(X^h,W^h),Y^v\rangle+
\frac{f_1\delta^2}{4}\sum\langle\calR(X^h,e_j),Y^v\rangle\langle\calR(W^h,e_j),Y^v\rangle\\
=\frac{f_1\delta^2}{4}\sum\langle\calR(X^h,e_j),Y^v\rangle\langle\calR(W^h,e_j),Y^v\rangle,  \hspace{1cm}
 \end{split}
\end{equation}
\begin{equation}
 R^G(X^v,Y^h,Y^h,W^h)=R^G(W^h,Y^h,Y^h,X^v)=\frac{f_2}{2}\langle(\na_{Y^h}\calR)(W^h,Y^h),X^v\rangle,\ \ \
\end{equation}
\begin{equation}
  R^G(X^v,Y^h,Y^h,W^v)=\frac{f_2\delta}{4}\sum_j\langle\calR(Y^h,e_j),W^v\rangle\langle\calR(Y^h,e_j),X^v\rangle,
\end{equation}
\begin{equation}
R^G(X^v,Y^v,Y^v,W^h)=0,  \hspace{1cm}
\end{equation}
\begin{equation}
 \begin{split}
  R^G(X^h,Y^h,Y^h,W^v)=\frac{f_2}{2}\langle(\na_{Y^h}\calR)(X^h,Y^h),W^v\rangle
\end{split}
\end{equation}
and of course $R^G(X^h,Y^v,Y^v,W^v)=0$. The simplification in formula
(\ref{formula70etal}) is due to property 2 in Proposition
\ref{partesdacurvatura} and the skew-symmetries of $R$. Henceforth the Ricci
curvature of $G$, the trace of the Ricci endomorphism, is given by
\begin{equation}
 \begin{split}\label{Riccihorhor}
  \ric^G(X^h,Y^h)\hspace{5cm}\\
=\sum_{i=1}^mR^G(X^h,\dfrac{e_i}{\sqrt{f_1}},\dfrac{e_i}{\sqrt{f_1}},Y^h)
+R^G(X^h,\dfrac{\theta e_i}{\sqrt{f_2}},\dfrac{\theta e_i}{\sqrt{f_2}},Y^h)\ =\ \ric(X^h,Y^h)  \hspace{.6cm}\\ 
-\frac{3}{4}\delta\sum_{j=1}^m\langle\calR(X^h,e_j),\calR(Y^h,e_j)\rangle
+\frac{\delta}{4}\sum_{i,j=1}^m\langle\calR(X^h,e_j),\theta e_i\rangle\langle\calR(Y^h,e_j),\theta e_i\rangle\\
=\:\ric(X^h,Y^h)-\frac{\delta}{2}\sum_{j=1}^m\langle\calR(X^h,e_j),\calR(Y^h,e_j)\rangle,
\end{split}
\end{equation}
\begin{equation}\label{Ricciverver}
\ric^G(X^v,Y^v)=\frac{\delta^2}{4}\sum_{i,j=1}^m\langle\calR(e_i,e_j),X^v\rangle\langle\calR(e_i,e_j),Y^v\rangle,
\end{equation}
\begin{equation}\label{Riccihorver}
  \ric^G(X^h,Y^v)=-\frac{\delta}{2}\sum_{i=1}^m\langle(\na_i\calR)(e_i,X^h),Y^v\rangle.
\end{equation}
And the scalar curvature is
\begin{equation}\label{curvaturaescalarTM}
 \begin{split}
 S^G=\sum_{k=1}^m\frac{1}{f_1}\ric^G(e_k,e_k)+\frac{1}{f_2}\ric^G(\theta e_k,\theta e_k)\qquad\\
=\frac{S}{f_1}-\frac{f_2}{4f_1^2}\sum_{i,j,k=1}^{m}(\calR_{ijk})^2
 \end{split}
\end{equation}
where $\calR_{ijk}=\langle\calR(e_i,e_j),\theta e_k\rangle=\langle R(e_i,e_j)u,e_k\rangle$ on each point $u\in TM$. Of course, $\ric$ and $S$ above denote respectively the Ricci and scalar curvatures of $M$. 

The following result generalises another from \cite{Seki} strictly for the Sasaki metric.
\begin{prop}
The Riemannian manifold $(TM,G)$ is Einstein $\Leftrightarrow$ $TM$ is flat $\Leftrightarrow$ $M$ is flat.
\end{prop}
\begin{proof}
If $TM$ is Einstein then $S^G$ is constant. In the present case it has a quadratic part varying in $\|u\|$, unless all $\calR_{ijk}=0,\ \forall u$.
\end{proof}
It is worth recalling the following results. The Sasaki metric of $TM$ is locally symmetric if and only if $M$ is flat (\cite{Kow1}). And, regarding what we continue studying next, the tangent unit sphere bundle is locally symmetric if and only if $(M,g)$ is flat or locally $(S^2(1),g_{\mathrm{std}})$.  Conformally flat is stronger: reserved for the locally standard 2-sphere (cf. \cite{Blair}). More recently it was proved semi-symmetric is the same as locally symmetric (\cite{BoeCal}).

\subsection{The second fundamental form of $S_rM$ and the Ricci and scalar curvature}

Let us start by recalling the theory of the second fundamental form of a Riemannian embedding. Suppose $Q^q$ is a submanifold of a Riemannian manifold $(N^{q+p},G)$ and $Q$ inherits the induced metric from $N$. Let $\na'$ denote the Levi-Civita connection of $N$ and let $X,Y$ be two vectors tangent to $Q$. Then we have the Gauss formula
\begin{equation}
 \na'_XY=\na_XY+\alpha(X,Y)
\end{equation}
where the sum respects the orthogonal decomposition $TQ\oplus TQ^\perp$. Passed the formality, $\na_XY$ is the Levi-Civita connection of $Q$. The clearly symmetric tensor
\begin{equation}
 \alpha:\Omega^0(TQ\otimes TQ)\lrr \Omega^0(TQ^\perp)
\end{equation}
is called the second fundamental form. Its trace $H^\alpha$ is the mean curvature vector. Let $\eta\in\Omega^0(TQ^\perp)$. Then we have the Weingarten formula $\na'_X\eta=-A_\eta X+D_X\eta$ where $A_\eta$ is a self-adjoint tensor on $TQ$ since $\langle A_\eta X,Y\rangle=-G(\na'_X\eta,Y)=G(\eta,\na'_XY)=G(\eta,\alpha(X,Y))$
and $D_X\eta$ is a metric connection on $TQ^\perp$. Finally we have the Gauss equation 
\begin{equation}\label{gaussequation}
 R(X,Y,Z,W)=R'(X,Y,Z,W)-G(\alpha(X,Z),\alpha(Y,W))+G(\alpha(Y,Z),\alpha(X,W)).
\end{equation}

We now resume with the study of the induced metric $G=g^{f_1,f_2}$ on the tangent sphere bundle $S_rM$ with radius function $r\in\cinf{M}$, with $\na=\nag$ and $f_1,f_2$ constant. Recall $m=n+1$ is the dimension of $M$.
\begin{prop}
$TS_rM=\{X\in TM:\ \langle X,\xi\rangle=rX(r)\}$.
\end{prop}
\begin{proof}
Indeed we have $\langle\xi,\xi\rangle-r^2=0$ defining the submanifold. Differentiating,
\begin{equation*}
X(\langle\xi,\xi\rangle-r^2)=2\langle\na^*_X\xi,\xi\rangle-2rX(r)=2(\langle X^v,\xi\rangle-rX(r))
\end{equation*}
we find the tangent space.
\end{proof}
In order to write the second fundamental form, we may write $\alpha$ as a scalar tensor:
\begin{equation}\label{alpha}
 \alpha(X,Y)=G(\na^G_XY,U^G)
\end{equation}
with $U^G$ a unit vector field defined on $S_rM$ and such that $U^G\perp^GTS_rM$. Writing
\begin{equation}
 U^G=a\grad r+b\xi
\end{equation}
for some functions $a,b$, we find the solution
\begin{equation}
 a=-\delta br\qquad\qquad\mbox{and}\qquad\qquad b=\frac{1}{r\sqrt{f_2+\delta f_2\tau^2}}
\end{equation}
where $\delta=f_2/f_1$ and $\tau=\|\grad r\|$.
\begin{prop}
The second fundamental form of $S_rM\subset TM$ with the induced metric $g^{f_1,f_2}$ and where $f_1,f_2$ are constants, is given by
\begin{equation}
\alpha(X,Y)=af_1(A(X,Y)(r)-\langle Y,\na_X\grad r\rangle)+bf_2(X(r)Y(r)-\langle Y^v,X^v\rangle).
\end{equation}
If $\na\dx r=0$, then the mean curvature is $H^\alpha=-\frac{n}{r\sqrt{f_2+\delta f_2\tau^2}}$.
\end{prop}
\begin{proof}
Continuing from (\ref{alpha}),
\begin{eqnarray*}
 \lefteqn{\alpha(X,Y)\ =}\\ &=& f_1\langle\na_XY^h+A(X,Y),a\grad r\rangle+f_2\langle\na_XY^v-\frac{1}{2}\calR(X,Y),b\xi\rangle\\
&=& af_1\langle \na_XY^h+A(X,Y),\grad r\rangle+bf_2\langle\na_XY^v,\xi\rangle\\
&=& af_1(X(Y(r))-\langle Y,\na_X\grad r\rangle+af_1A_{X,Y}(r)+bf_2(X(rY(r))-\langle Y^v,\na_X\xi\rangle)\\
&=& (af_1+bf_2r)X(Y(r))+af_1(A_{X,Y}(r)-\langle Y,\na_X\grad r\rangle)+bf_2(X(r)Y(r)-\langle Y,X^v\rangle)
\end{eqnarray*}
and the result follows. For the mean curvature we take a horizontal $g$-orthonormal frame $e_1,\ldots,e_m$ with $e_m=u/r$. Then the $Y_i=\frac{1}{\sqrt{f_2}}\theta e_i$ for $i=1,\ldots,n$ constitute a vertical frame tangent to $S_rM$. There must also exist an extension of these vectors to an o.n. frame of $T_uS_rM$, and therefore a $m\times m$-matrix $a_{ip}\in\R$ inducing $m$ vectors $X_i=\sum_pa_{ip}e_p+x_i\frac{\xi}{r}$, tangent and o.n. to each other and to the $Y_j$; in particular with $x_i=X_i(r)\in\R$. Now the condition $\na\grad r=0$ implies $A(X,Y)(r)=0$ for all $X,Y$ because in the definition we find the symmetrization of 
\begin{eqnarray*}
\langle\calR(X,\grad r),Y\rangle=-\langle R(u,\theta^tY)\grad r,X^h\rangle=0.
\end{eqnarray*}
Finally,
\begin{eqnarray*}
H^\alpha &=& \sum_{i=1}^m\alpha(X_i,X_i)+\sum_{j=1}^n\alpha(Y_j,Y_j) \\ 
&=& \sum bf_2(X_i(r))^2-bf_2x_i^2-\sum_jb \ =\ -nb
\end{eqnarray*}
\end{proof}
So one has the formulas to compute the Riemannian curvature $\tilde{R}$ of $S_rM$.

From now on we assume $r$ is a constant. Then
\begin{equation}\label{formulade_abalpha}
 b=\frac{1}{r\sqrt{f_2}},\qquad a=-\frac{\sqrt{f_2}}{f_1}\qquad\mbox{and}\qquad
\alpha(X,Y)=-\frac{\sqrt{f_2}}{r}\langle X^v,Y^v\rangle.
\end{equation}
Henceforth, by Gauss formula (\ref{gaussequation}), the curvature ${\tilde{R}}^G(X,Y,Z,W)$ does not differ from that one, given previously for the ambient manifold, except if all four vectors are vertical. Minor adaptations must follow in the Ricci and scalar curvatures, respectively  ${\tilde{\ric}}^G$ and ${\tilde{S}}^G$, of the tangent sphere bundle.
\begin{prop}
With $\ric^G$ and $S^G$ restricted to $S_rM$, we have
\begin{meuenumerate}\label{curvaRiccieescalardeS_rM}
 \item ${\tilde{\ric}}^G=\ric^G+\frac{n-1}{r^2}g_{|_{V\otimes V}}$.
 \item ${\tilde{S}}^G=S^G+\frac{(n-1)n}{f_2r^2}$
\end{meuenumerate}
\end{prop}
\begin{proof}
The fibres are $n$-dimensional spheres. The differences ${\tilde{\ric}}^G-\ric^G$ and ${\tilde{S}}^G-S^G$ are easy to check from (\ref{formulade_abalpha}) and the Gauss equation. More closely
\begin{eqnarray*}
 {\tilde{\ric}}^G(X,Y) &=& \ric^G(X,Y)+
\frac{1}{f_2}\sum_{i=1}^n{\tilde{R}}^G(X^v,\theta e_i,\theta e_i,Y^v)\\
&=& \ric^G(X,Y)+\frac{1}{f_2}\sum\bigl(-\alpha(X,\theta e_i)\alpha(\theta e_i,Y)+\alpha(\theta e_i,\theta e_i)\alpha(X,Y)\bigr)\\
&=& \ric^G(X,Y)+\frac{n}{r^2}\langle X^v,Y^v\rangle-\frac{1}{r^2}\langle X^v,Y^v\rangle.
\end{eqnarray*}
Looking at formula (\ref{Riccihorhor}), we see the sum in $i$ of the $R^G(X,\theta e_i,\theta e_i,Y)$ up to $m=n+1$ gives the same as the sum up to $n$. This is because we may take an orthonormal basis of $V$ at each point $u$ such that $u/r$ is the last vector and then we notice $\langle\calR(X^h,e_j),\xi\rangle=0$. Recall $u\perp T_uS_rM$ and $\xi_u=u$. The same question is not put in formulas (\ref{Ricciverver},\ref{Riccihorver}). The same observations are made for ${\tilde{S}}^G$.
\end{proof}
\begin{teo}\label{enfimcurvaturaescalarpositiva}
Let the radius $r$ be a fixed constant. We have the following:
\begin{meuenumerate}
\item For a surface $M$ the bundles $TM$ and $S_rM$ have the {\em same} Ricci and scalar curvatures.
\item Let $m\geq3$ and suppose $M$ has bounded sectional curvatures (e.g. if it
is compact). Then:
\begin{enumerate}
\item for any $f_2$ there exists a sufficiently large $f_1$ such
that the tangent sphere bundle $(S_rM,g^{f_1,f_2})$ has positive scalar
curvature.
\item for any $f_1$ there exists a sufficiently small $f_2$ such
that the tangent sphere bundle $(S_rM,g^{f_1,f_2})$ has positive scalar
curvature.
\end{enumerate}
\end{meuenumerate}
\end{teo}
\begin{proof}
It is clear by a polarization process that all values $\calR_{ijk}$ in formula
(\ref{curvaturaescalarTM}) remain bounded on $S_rM$. The result follows
combining with Proposition \ref{curvaRiccieescalardeS_rM}.
\end{proof}
In the present setting, we immediately generalise Theorems 1 and 2 in \cite{KowSek2}.
\begin{teo}[\cite{KowSek2}]\label{generalizacaodeKowSek}
Let $\dim M\geq3$ and suppose $M$ has bounded sectional curvatures (e.g. if it is compact). Then the tangent sphere bundle $(S_rM,g^{f_1,f_2})$ has positive scalar curvature for all sufficiently small constant radius $r>0$.
\end{teo}
We just remark that \cite[Theorem 2]{KowSek2} essentially gives conditions for achieving \textit{negative} scalar curvature. We may state analogous result for the weighted metric.

\vspace{1.5cm}

%\medskip

%\end{color}

\end{document}